\documentclass[12pt,oneside]{amsart}

\usepackage{pictexwd,dcpic}
\usepackage{amssymb}
\usepackage{amsmath}
\usepackage{amsthm}

\theoremstyle{plain}
\newtheorem{teo}{Theorem}[section]
\newtheorem{cor}[teo]{Corollary}
\newtheorem{prop}[teo]{Proposition}

\theoremstyle{definition}

\theoremstyle{remark}

\makeatletter
\def\@setcopyright{}
\def\serieslogo@{}
\makeatother

\newcommand{\XX}{\mathcal{X}}
\newcommand{\OS}{\mathcal O_{\!S}}
\newcommand{\OSU}{\mathcal{O}_{\!S}^{\:\!*}}
\newcommand{\PP}[1]{\mathbb P_{#1}}

\newcommand{\funz}[5]{\begin{array}{cccc}#1\colon&#2&\longrightarrow&#3\\
 &#4&\longmapsto&#5\end{array}}

\begin{document}

\begin{center}
{\LARGE Integral points on the complement\\ of the branch locus of projections from hypersurfaces.

}\vspace{.8cm}

\textsc{Andrea Ciappi}

(Universit\`a di Pisa, Italy)

\vspace{1.5cm}

\thispagestyle{empty}

\begin{minipage}[b]{.8\textwidth}

\begin{footnotesize}\textsc{Abstract. }We study the integral points on $\PP n\setminus D$, where $D$ is the branch locus of a projection from an hypersurface in $\PP{n+1}$ to a hyperplane $H\simeq\PP n$. In doing that we follow the approach proposed in a paper by Zannier but we prove a more general result that also gives a sharper bound that may lead to prove the finiteness of integral points and has more applications. The proofs we present in this paper are effective and they provide a way to actually construct a set containing all the integral points in question. Our results find a concrete application to Diophantine equations, more specifically to the problem of finding integral solutions to equations $F(x_0,\dots,x_n)=c$, where $c$ is a given nonzero value and $F$ is a homogeneous form defining the branch locus $D$.
\end{footnotesize}
\end{minipage}

\end{center}\vspace{.5cm}

\vspace{.3cm}

\vspace{.5cm}

\section{Introduction}

The study of integral points\footnote{For the definition of integral points (and the concept of sets of $S$-integral points and quasi-$S$-integral points) we rely on \cite{ser} and \cite{zzz}} on varieties defined in a projective space as the complement of certain divisors is related to several Diophantine problems and it is a recurring and interesting problem in number theory.

Given an affine variety $V\subset\PP n$, we can consider its closure $\overline V$ in $\PP n$ and its divisor at infinity $D=\overline V\setminus V$. Many valuable thorems about integral points on $V$ have been proved in the last century, but the majority of them require the splitting of the divisor $D$ in several components in order to be applied. There is a standard technique to bypass this requirement that consists in lifting integral points by means of a finite cover of the variety $\overline V$ unramified except possibly above points in $D$ and such that the pull-back of $D$ has more components than $D$ itself. However, this method seldom apply if $\dim V>1$ because, in general, the pull-back of $D$ does not split as desired.

A remarkable exception to this is a result by Faltings, who proved the finiteness of integral points on the complements of certain irreducible singular curves in $\PP 2$. In this case, the divisor $D$ is the branch locus of a suitable projection from a smooth surface described in detail in the original paper \cite{fal}. The problem was also studied by Zannier who proved a similar result in \cite{zan} applying arithmetic considerations from \cite{cz} (and hence ultimately relying on the Schmidt's Subspace Theorem) to the same geometric setting introduced by Faltings. Zannier obtained the same conclusions under different hypotheses and, moreover, he proved that the fact that the projected surface has non-negative Kodaira dimension is a sufficient condition for the finiteness of integral points on $\PP 2\setminus D$. Later, both results were improved by Levin in \cite{lev}, where the theorem is proved even for surfaces with negative Kodaira dimension.

In \cite{zan} Faltings' principle is also applied to the simpler case of a projection taken from an hypersurface in $\PP{n+1}$ and it leads to a bound for the dimension of any set of integral points on the complement in $\PP n$ of the branch locus of the projection. The analysis presented here will be similar but more general, as we will make no assumption on the projection. This will require some more care but it will also lead to stronger conclusions and more applications.

The geometric setting of the problem is described in detail in the second section of this paper along with the statement of our main result, proved in the third section, while the last part contains some remarks and corollaries. On a final note, we observe that the results we present (which are effective, see Section 4) have a concrete application to the study of Diophantine equations $F(x_0,\dots,x_n)=c$ for certain homogeneous irreducible forms $F$ and non-zero values $c$ (see Proposition \ref{xyz}).\vspace{.5cm}

\section{Setting of the problem}

Let $k$ be a number field and $S$ a finite set of places of $k$ which includes all the infinite ones. Let $\XX$ and $H$ be, respectively, an irreducible hypersurface of degree $m>1$ and a hyperplane in the projective space $\PP{n+1}$, both defined over $k$. Let $Q$ be a point in $\PP{n+1}\setminus H$ and consider the projection of $\XX$ from the point $Q$ to $H$; we shall denote by $\phi$ the projection and by $D\subset H\simeq\PP{n}$ its branch locus. 

Without any loss of generality we suppose $Q=(0:\ldots:0:1)$ and that $H$ is defined by $X_{n+1}=0$. The projection $\phi$, takes then the form
\[
\funz{\phi}{\XX}{\PP{n}}{(x_0:\ldots:x_n:x_{n+1})}{(x_0:\ldots:x_n).}
\]
If $Q\in\XX$ then $\phi(Q)$ is not defined unless we consider a blow-up. However, for our purposes, it will suffice to consider the restriction $\phi_{|\XX\setminus Q}$ which, with a slight abuse of notation, will still be denoted by $\phi$.

Let $f\in k[X_0,\ldots,X_{n+1}]$ be a homogeneous irreducible polynomial of degree $m$ defining $\XX$; we may view it as a univariate polynomial in $X_{n+1}$ with coefficients in $k[X_0,\ldots,X_n]$
$$f(X_0,\ldots,X_n,X_{n+1})=\sum_{l=0}^d f_l(X_0,\ldots,X_n)X_{n+1}^{d-l},$$
where $d=\deg_{X_{n+1}}f$ is the greatest integer such that the coefficient of $X_{n+1}^d$ is not identically zero and every $f_l$ is a homogeneous polynomial of degree $m-d+l$ (or the null polynomial). We remark that the geometrical request $Q\notin\XX$ implies $d=m$, $\textrm{deg }\!\!\: f_l=l$ and $f_0\in k^*$ (this being the case discussed in \cite{zan}).

We consider the discriminant of $f$ in respect of $X_{n+1}$, a polynomial in\linebreak $X_0,\ldots,X_n$ that we shall denote by $\Delta=\Delta(X_0,\ldots,X_n)$. Its zeroes are exactly the ramification points of $\phi$, insofar as $Q$ does not belong to $\XX$, and in this case $\Delta =0$ is the defining equation for $D$. On the other hand, if $Q\in\XX$, there are points in $\PP{n}$ where the polynomial $f_0$ vanishes: they may or may not belong to $\phi(\XX)$ or to $\{\Delta =0\}$, but their preimages under $\phi$ surely have cardinality less than $d$. Hence the branch locus $D$ is defined as the union of the zero loci of $f_0$ and $\Delta$. 

We also consider a set $T$ composed by the points $(x_0:\ldots:x_n)\in\PP{n}$ such that $f(x_0,\ldots,x_n,X)$ has exactly one root or none at all. If, for example, we require one root with multiplicity $d$, we must have $f(x_0,\dots,x_n)\neq 0$ and we look for a factorization
\[
\sum_{l=0}^d f_l(x_0,\dots,x_n)X^{d-l}=f_0(x_0,\dots,x_n)\cdot(X-\alpha)^d
\]
where $\alpha=\alpha(x_0,\dots,x_n)\in\bar{k}$ and we turn it in $d$ equations
\[
f_l(x_0,\dots,x_n)=f_0(x_0,\dots,x_n)\cdot{d\choose l}(-\alpha)^{\, l}\qquad l=1,\ldots,d.
\]
In particular, we must have $f_1\:\!=\:\!-d\alpha f_0$ or, equivalently, $-\alpha\:\!=\:\! f_1/(d\:\! f_0)$. This leads to the following relations among the polynomials:
\begin{equation}\label{t0}
\begin{split}
& f_0\neq 0\\
& f_l={d \choose l}\frac{f_1^{\, l}}{d^{\, l}f_0^{\, l-1}}\qquad\forall\ l=2,\ldots,d.
\end{split}
\end{equation}
We denote by $T_0$ the set of points satisfying the above relations. We define in an analogous way the sets $T_1,\ldots ,T_{d-1}$ consisting, respectively, of the points in $\PP{n}$ whose preimages \emph{via} $\phi$ are made by single points with multiplicity, respectively, $d-1,d-2,\ldots,1$. For example, the points in $T_1$ will satisfy $f_0=0$, $f_1\neq 0$ and
\[
f_l={d-1 \choose l-1}\frac{f_2^{\, l-1}}{(d-1)^{\, l-1}f_1^{\, l-2}}\qquad\forall\ l=3,\ldots,d.
\]
Finally, we have $T=T_0\cup\dots\cup T_d$, where the last two sets involved are $T_{d-1}=\{f_0=\dots =f_{d-2}=0,f_{d-1}\neq 0\}$ and the set of points not belonging to $\phi(\XX)$, $T_d=\{f_0=\ldots =f_{d-1}=0\ ,\ f_d\neq 0\}$.

We can now state our result:
\begin{teo}\label{main}
Assuming the hypotheses and notations discussed above in this section, the Zariski closure of any set of quasi-$S$-integral points for $\PP n\setminus D$ has dimension less than or equal to $\dim T_0+1$.
\end{teo}\vspace{.5cm}

\section{Proof}

We will make use of the following well-known fact (for a proof, see Proposition 2.3 in \cite{zan}):

\begin{prop}\label{xyz}
Let $L\subset\PP n$ be an effective divisor defined by a form $\Lambda\in k[X_0,\ldots,X_n]$ and let $\Sigma$ be a set of quasi-$S$-integral points for the affine variety $\PP n\setminus L$. Then there exists a finite set of places $S'\supset S$ of $k$ such that each point of $\Sigma$ has projective coordinates $(x_0:\dots:x_n)$ with $x_i\in\mathcal O_{\!S'}$ and $\Lambda(x_0,\ldots,x_n)\in\mathcal{O}_{\!S'}^{\:\!*}$.
\end{prop}

In order to give more emphasis to the underlying methods and ideas leading to the result, we postpone the discussion of the ``low degrees'' case. More precisely, during the proof we will make the assumption that the degree $d$ of the polynomial $f_0(X_0,\dots,X_n)$ is greater than or equal to $4$. We will go back to that point in the next section (paragraph \emph{``Low degrees''}) and complete the proof for $d=2$ and $d=3$.\vspace{.2cm}

\noindent\textsc{First step -} Let $\Sigma$ be a set of quasi-$S$-integral points for $\PP{n}\setminus D$. By the above Proposition there exists a finite set $S'\supset S$, made up of places of $k$, such that for every point in $\Sigma$ there are projective coordinates $(x_0:\ldots:x_n)$ such that every $x_i$ belongs to $\mathcal O_{\!S'}$ and $\Delta(x_0,\dots,x_n)\in\mathcal{O}_{\!S'}^{\:\!*}$; we choose $P\in\Sigma$ and projective coordinates $(x_0:\dots:x_n)$ for it so that the properties we just mentioned are satisfied.

Then we consider the equation $f(x_0,\ldots,x_n,X)=0$ which has $d$ distinct roots in $\overline{\mathbb Q}$ since $P\notin D$. We shall denote them by $\alpha_1,\ldots,\alpha_d$ and we consider the number field $k'$ they generate over $k$, which depends on $P$: it has bounded degree and it is unramified except at places above $S'$. Hermite's Theorem implies that there are at most a finite number of number fields with these properties, hence we may choose a number field $k''$ such that it contains all the roots $\alpha_i$ regardless of the chosen point $P$. Finally, we may define a finite set $S''$ constituted by places of $k''$ that contains the extension of $S'$ to a set of places of $k''$ and such that the polynomials $f_i(X_0,\ldots,X_n)$ have coefficients in $\mathcal{O}_{k'',\ S''}$. Noting that by enlarging $k$ or $S$ we make our conclusion stronger, we assume in this proof $k=k''$ and $S=S''$.\vspace{.2cm}

\noindent\textsc{Second step -} We can now consider the usual factorization of the discriminant $$\Delta(x_0,\ldots,x_n)=f_0^{2d-2}\prod_{1\leqslant i<j\leqslant d}(\alpha_i -\alpha_j)^2$$ which is valid because $f_0(x_0,\ldots,x_n)\neq 0$ since $P\notin D$. Every root can be written as a product $\alpha_i=\mu_i\delta_i^{\:-1}$ with $\mu_i$ and $\delta_i$ coprime $S$-integers. We also note that every polynomial $\delta_i X-\mu_i$ divides $f(x_0,\ldots,x_n,X)$ in $\OS[X]$, thus $\delta_1\cdots\delta_d$ divides $f_0$ in $\OS$. It follows that $\Delta(x_0,\ldots,x_n)$ is divisible in $\OS$ by $\prod_{i\neq j}(\delta_j\mu_i -\delta_i\mu_j)$ and, since the discriminant is an $S$-unit, we deduce that every factor $\delta_j\mu_i -\delta_i\mu_j$ must be in $\OSU$.

We define $x_{ij}:=\delta_j\mu_i - \delta_i\mu_j$ and we consider the identity $$x_{i1}x_{23}+x_{i2}x_{31}+x_{i3}x_{12}=0$$ where $i\in\{4,\ldots,d\}$ and every summand is clearly in $\OSU$. Since we just produced solutions to the homogeneous $S$-unit equation, we may apply some finiteness result (see \cite{sch} or \cite{zzz}) and obtain that, for example, the ratio $x_{i2}x_{31}/x_{i1}x_{32}$ lies in a finite set independent of the chosen point $P$. In order to write down some algebraic relations among the roots $\alpha_i$, we observe that we have just proved that for certain $c_i=c_i(P)$ in a fixed finite set, we have $$c_i=\frac{x_{i2}x_{31}}{x_{i1}x_{32}}=\frac{(\alpha_i -\alpha_2)(\alpha_3 -\alpha_1)}{(\alpha_i -\alpha_1)(\alpha_3 -\alpha_2)}\qquad i\in\{4,\ldots,d\}$$
and if we put $c_2:=0$ and $c_3:=1$ we have analogous relations for $i=2$ and $i=3$. After some easy manipulations, we can write the following expressions for the roots:
\begin{equation}\label{alphai}\alpha_i =\left\{
\begin{array}{ll}
\frac{\displaystyle\alpha_1 (\alpha_2-\alpha_3)c_i+\alpha_2(\alpha_3-\alpha_1)}{\displaystyle(\alpha_2-\alpha_3)c_i+\alpha_3-\alpha_1} & i=2,\ldots,d\\
\\
\frac{\displaystyle\alpha_4 (\alpha_2-\alpha_3)c_4+\alpha_3(\alpha_4-\alpha_2)}{\displaystyle(\alpha_2-\alpha_3)c_4+\alpha_4-\alpha_2} & i=1.\\
\end{array}\right.
\end{equation}
Finally, we can split $\Sigma$ into finitely many subsets such that the $c_i$'s are fixed for every point in a given subset. Arguing separately with each subset we may then assume that the $c_i$'s do not depend on $P$.\vspace{.2cm}

\noindent\textsc{Third step -} We pause to outline how we will make use of the information obtained so far. We are going to define a quasi-projective variety in $\PP{n+4}$ and its projection on $\PP n$ will lead to the sought relation between $\Sigma$ and $T_0$. Intuitively, $n+1$ coordinates are required to define a point in $\Sigma\subset\PP n$ and four values are required to express all the roots $\alpha_i$, see (\ref{alphai}) above. The polynomials that we are about to introduce are defined following the relations (\ref{alphai}) and then considering Vi\`ete's formulae to provide a link between the roots $\alpha_i$ and the polynomials $f_i$: they are essential in the definition of the quasi-projective variety above mentioned.

We start defining some auxiliary polynomials in $k[Y_1,Y_2,Y_3,Y_4]$:
\[
a_i(Y_1,Y_2,Y_3,Y_4) =\left\{
\begin {array}{ll}
Y_1(Y_2-Y_3)c_i+Y_2(Y_3-Y_1) &\quad i=2,\ldots,d\\
\\
Y_4(Y_2-Y_3)c_4+Y_3(Y_4-Y_2) &\quad i=1\\
\end{array}\right.
\]
\[
b_i(Y_1,Y_2,Y_3,Y_4) =\left\{
\begin {array}{ll}
\phantom{Y_1}(Y_2-Y_3)c_i+\phantom{Y_2(}Y_3-Y_1\phantom{)} &\quad i=2,\ldots,d\\
\\
\phantom{Y_4}(Y_2-Y_3)c_4+\phantom{Y_3(}Y_4-Y_2\phantom{)} &\quad i=1\\
\end{array}\right.
\]
\[
\begin{array}{rcl}A_l(Y_1,Y_2,Y_3,Y_4) &\! =\! & \displaystyle{\sum_{1\leqslant i_1<\ldots <i_l\leqslant d}\! a_{i_1}(Y_1,Y_2,Y_3,Y_4)\cdots a_{i_l}(Y_1,Y_2,Y_3,Y_4)}\ \cdot\\
\\
&\cdot & \displaystyle{\prod_{\substack{1\leqslant j\leqslant d\\j\neq i_1,\ldots,i_l}}b_j(Y_1,Y_2,Y_3,Y_4)}\hspace{1.9cm} l=1,\ldots ,d
\end{array}
\]
\[
B(Y_1,Y_2,Y_3,Y_4)=\prod_{i=1}^d b_i(Y_1,Y_2,Y_3,Y_4).
\]

If, as before, $P=(x_0:\ldots:x_n)\in\Sigma$ is the point in question and $f(x_0,\ldots,x_n,X)$ has roots $\alpha_1,\ldots,\alpha_d$, we observe that, because of (\ref{alphai}),
\begin{equation}\frac{a_i(\alpha_1,\alpha_2,\alpha_3,\alpha_4)}{b_i(\alpha_1,\alpha_2,\alpha_3,\alpha_4)}=\alpha_i\qquad i=1,\ldots,d.\end{equation}
Furthermore, since the coefficients of a polynomial can be expressed as the product of the leading coefficient and the correspondent symmetric function calculated in its roots, we have for $l=1,\ldots,d$
\begin{equation}\label{fl}
f_l(x_0,\ldots,x_n)=(-1)^{\, l} f_0(x_0,\ldots,x_n)\cdot\frac{A_l(\alpha_1,\alpha_2,\alpha_3,\alpha_4)}{B(\alpha_1,\alpha_2,\alpha_3,\alpha_4)}.
\end{equation}
After these remarks, we are ready to define and study a projective variety $V\subset\PP{n+4}$ given by the common zero locus of the $d$ polynomials
\begin{equation}\label{defv}B f_l -(-1)^{\, l} f_0 A_l\end{equation}
where $B$ and $A_l$ belong to $k[Y_1,Y_2,Y_3,Y_4]$, $f_0$ and $f_l$ are in $k[X_0,\dots, X_n]$ and $l$ ranges from $1$ to $d$.

Since our main interest is focused on $\Sigma\subset\PP n$, we are going to consider the projection of $V$ to $\PP n$ by taking the first $n+1$ coordinates. To ensure that the projection is well-defined we have to remove points with nothing but zeroes in the first $n+1$ coordinates. In addition, we are going to ignore points that belong to $V$ regardless of the first $n+1$ coordinates, only because the $Y_i$'s have special values. Furthermore, we would like, at some point, to get rid of the zeroes of $f_0(X_0,\dots,X_n)$ in $\PP n$, because they cannot be in $T_0$. We accomplish these goals by defining the varieties
$$\begin{array}{ll}U_0:=\{&\!\!\!\!\!(z_0:\ldots:z_n:y_1:y_2:y_3:y_4)\in V:z_0=\dots =z_n=0\}\vspace{.15cm}\\
U_1:=\{&\!\!\!\!\!(z_0:\ldots:z_n:y_1:y_2:y_3:y_4)\in V:B(y_1,y_2,y_3,y_4)=0\textrm{ and}\vspace{.05cm}\\
&\!\!\!\!\!A_l(y_1,y_2,y_3,y_4)=0\quad\forall l=1,\dots, d\}\vspace{.15cm}\\
U_2:=\{&\!\!\!\!\!(z_0:\ldots:z_n:y_1:y_2:y_3:y_4)\in V:f_0(z_0,\dots,z_n)=0\}\vspace{.15cm}\\
U:=&\!\!\!\!\!\!\!U_0\cup U_1\cup U_2\end{array}$$
and a quasi-projective variety which is the complement of $U$ in $V$:
$$W:=V\setminus U.$$
Finally, we consider the projection from W to $\PP n$:
\[
\funz{\pi}{W}{\PP{n}}{(z_0:\ldots:z_n:y_1:y_2:y_3:y_4)}{(z_0:\ldots:z_n).}
\]\vspace{.2cm}

\noindent\textsc{Fourth step -} Once again we look at $P=(x_0:\ldots:x_n)\in\Sigma$ and we observe that $(x_0:\ldots:x_n:\alpha_1:\alpha_2:\alpha_3:\alpha_4)\in V$ because of \eqref{fl}. We also observe that there exists $i\in\{1,\dots,d\}$ such that $x_i\neq 0$ since $P\in\PP n$, hence $P\notin U_0$. Furthermore, $f_0(x_0,\dots,x_n)\neq 0$ because $P\notin D$ and $B(\alpha_1,\alpha_2,\alpha_3,\alpha_4)\neq 0$ because the roots $\alpha_i$ are all pairwise distinct. It follows that $(x_0:\ldots:x_n:\alpha_1:\alpha_2:\alpha_3:\alpha_4)$ actually belongs to $W$, whence $\Sigma\subset\pi(W)$.

We investigate now what happens to $W$ when intersected with the hyperplane $\{Y_2=Y_3\}\subset\PP{n+4}$. First of all we notice that
$$\frac{a_i(y_1,y_2,y_2,y_4)}{b_i(y_1,y_2,y_2,y_4)}=y_2\qquad i=1,\ldots,d$$
and, subsequently, we have
$$\frac{A_l(y_1,y_2,y_2,y_4)}{B(y_1,y_2,y_2,y_4)}=\sum_{1\leqslant i_1<\ldots <i_l\leqslant d}\frac{a_{i_1}}{b_{i_1}}\cdots\frac{a_{i_l}}{b_{i_l}}={d\choose l}\ y_2^{\, l}\qquad l=1,\ldots,d.$$
From the defining equations of $V$ and the ones displayed above, we have for every point $(z_0:\ldots:z_n:y_1:y_2:y_2:y_4)\in W\cap\{Y_2=Y_3\}$ the following relation:
\begin{equation}\label{int}
\begin{split}
f_l(z_0,\ldots,z_n) & =(-1)^{\, l} f_0(z_0,\ldots,z_n)\frac{A_l(y_1,y_2,y_2,y_4)}{B(y_1,y_2,y_2,y_4)}\\
& =(-1)^{\, l} f_0(z_0,\dots,z_n) {d\choose l}y_2^{\, l}
\end{split}
\end{equation}
which is valid for $l=1,\ldots,d$. In particular, we get $f_1=-d\:\!f_0\;\!y_2$ and therefore
$$f_l(z_0,\ldots,z_n)=f_0(z_0,\ldots,z_n){d\choose l}\left(\frac{f_1(z_0,\ldots,z_n}{d\:\! f_0(z_0,\ldots,z_n)}\right)^{\, l}\qquad l=1,\ldots,d.$$
Then, recalling the defining equations for $T_0$ (\ref{t0}), we have just proved that $\pi\big(W\cap\{Y_2=Y_3\}\big)\subset T_0$.

We draw a diagram to help us clarify the role of the auxiliary objects we introduced in the proof:\vspace{.15cm}
\[
\begindc{\commdiag}[20]
\obj(5,3)[w]{$W\cap\{Y_2=Y_3\}$}
\obj(0,3)[wi]{$W$}
\obj(5,0)[pw]{$\pi\big(W\cap\{Y_2=Y_3\}\big)$}
\obj(0,0)[pwi]{$\pi(W)$}
\obj(10,0)[s]{$T_0$}
\obj(-3,0)[t]{$\Sigma$}
\mor{w}{wi}{}[\atright,\injectionarrow]
\mor{w}{pw}{$\pi$}
\mor{pw}{s}{}[-1,\injectionarrow]
\mor{t}{pwi}{}[\atright,\injectionarrow]
\mor{wi}{pwi}{$\pi$}
\mor{pw}{pwi}{}[\atright,\injectionarrow]
\enddc\vspace{.2cm}
\]
Finally, we consider the Zariski closure of $\Sigma$ and we readily have that $\dim\overline{\Sigma}\leqslant\dim T_0+1$. This completes the proof for $d\geqslant 4$.\qed\vspace{.5cm}

\section{Details and remarks}

\noindent\emph{Effectivity -} A noteworthy feature of Theorem \ref{main} is its effectivity. This is a consequence, essentially, of the fact that we obtained a finiteness result during the second step of the proof without the help of Schmidt's Theorem or other ineffective conclusions from Diophantine approximation. Instead, we used results about $S$-unit equations and it is known that a finite and complete set of non-proportional representatives can be effectively found (for example \emph{via} Baker's theory, see \cite{bak}). Therefore it is possible to determine all the auxiliary objects introduced in the proof, assuming $\XX$ is given, and we may actually exhibit the set $\pi(W)$ containing $\Sigma$.

We must point out that the set of solutions depends naturally on $k$ and $S$; they may have been enlarged with the application of Proposition \ref{xyz}, so an explicit notion of quasi-$S$-integral points is also required to have a unique determination for the solutions of the $S$-unit equation. In other words, we are required to specify an affine model for $\PP{n}\setminus D$.

We also remark that another result of crucial importance in our proof is Hermite's Theorem, which is effective as well.\\

\noindent\emph{Analysis of the results -} We would like to study the dimension of $T_0$, once the geometric setting is specified, and to compare it to the dimension of $T$ (which replaces $\dim T_0$ in the bound given in \cite{zan}). Obviously we have $\dim T_0\leqslant\dim T$, as $T_0\subset T$, but it is not hard to see that equality holds very often. In fact, $T$ is the disjoint union of its $d+1$ subsets $T_i$, each of them defined by an inequality and $d-1$ equations (save $T_d$ which is defined by $d$ equations) and $\dim T=\max\{\dim T_i\}_{i=0,\dots,d}$.

In order to study the difference between $\dim T$ and $\dim T_0$, it may be useful to have explicit conditions for the sets $T_i$. We see that a point $(x_0:\dots:x_n)\in\PP n$ belongs to $T_i$ if and only if the following conditions are satisfied (we denote $f_j(x_0,\dots,x_n)$ simply by $f_j$):
\begin{equation}\label{ti}\left\{
\begin{array}{ll}
f_l=0 & \forall\ l<i\\
f_i\neq 0&\\
(d-i)^{l-i}\ f_i^{l-i-1}\ f_l={d-i\choose l-i}f_{i+1}^{l-i}\qquad & \forall\ l\geqslant i+2.\\
\end{array}\right.
\end{equation}
We observe that if there is $l<i$ such that $f_l$ divides $f_i$ in $k[X_1,\dots,X_n]$ we have $T_i=\varnothing$. If we ask $Q\notin\XX$ we have $f_0\in k^*$ and so, for every $i=1,\dots,d$, we have $f_0|f_i$, whence $T_i=\varnothing$ and $T=T_0$.\\

\noindent\emph{A criterion for finiteness -} Suppose that there are $i,j\in\{1,\dots,d\}$ such that the polynomial $f_i$ is the null polynomial and $f_j$ vanishes only if $f_0$ does. Then, recalling conditions (\ref{ti}) for the set $T_0$, we get $f_0\neq 0$ and $f_1=0$; this happens trivially if $i=1$ and comes from the equation $d^if_0^{i-1}f_i=\binom{d}{i}f_1^i$ otherwise. This yields the condition $f_l=0$ for every $l\geqslant 1$, hence $f_j$ must vanish which contradicts the requirement $f_0\neq 0$. Hence $T_0=\varnothing$.

\begin{cor}\label{corfin}Notation being as in Section 2, suppose that there are $i,j\in\{1,\dots ,d\}$ such that $f_i(X_0,\dots,X_n)$ is the null polynomial and $f_j(X_0,\dots,X_n)=0$ implies $f_0(X_0,\dots,X_n)=0$. Then every set of quasi-$S$-integral points for $\PP n\setminus D$ is a finite set.
\end{cor}\vspace{.3cm}

\noindent\emph{On the complement of $\{\Delta=0\}$ -} We state and prove a Corollary of Theorem \ref{main} which allows for the points of $\PP{n}$ where the leading coefficient of $f(X_0,\dots,X_{n+1})$ as a polynomial in $X_{n+1}$ vanishes. In other words, we investigate the quasi-$S$-integral points on the complement of the divisor defined by the discriminant.

\begin{cor}\label{cordelta}Notations being as in Section 2, the Zariski closure of any set $\Sigma$ of quasi-$S$-integral points for $\PP{n}\setminus\{\Delta =0\}$ has dimension less than or equal to $\dim (T_0\cup T_1)+1$.\newline Moreover, if $f_0(x_0,\dots,x_n)\neq 0$ (resp. $f_0(x_0,\dots,x_n)=0$) for every $(x_0:\dots: x_n)\in\Sigma$, we have that the dimension of the Zariski closure of $\Sigma$ is less than or equal to $\dim T_0+1$ (resp. $\dim T_1+1$).\end{cor}

\begin{proof}
Let $\Sigma$ be a set of quasi-$S$-integral points for $\PP{n}\setminus\{\Delta =0\}$ and consider a point $P=(x_0:\dots:x_n)\in\Sigma$. As before, we look at the polynomial $f(x_0,\ldots,x_n,X)$ which has $d$ or $d-1$ roots: we denote these pairwise distinct roots by $\alpha_1,\ldots,\alpha_{d-1}$ and, in case, $\alpha_d$. The first thing we observe is that $f_0(x_0,\dots,x_n)$ and $f_1(x_0,\dots,x_n)$ cannot be both equal to zero for otherwise we would have $\Delta(x_0,\dots,x_n)=0$. Again, we can apply Proposition \ref{xyz} and enlarge $k$ and $S$ to ensure that every point of $\Sigma$ has projective coordinates with entries in $\OS$ and that $\Delta$ has values in $\OSU$.

If $f_0(x_0,\dots,x_n)\neq 0$ we follow the proof of Theorem \ref{main} until we get the relations \eqref{alphai} among the roots. If $f_0(x_0,\dots,x_n)=0$ we consider the discriminant $\Delta_{d-1}(\boldsymbol x)$ of the polynomial $f(x_0,\dots,x_n,X)\in k[X]$ of degree $d-1$ and we observe that
\[
\Delta(\boldsymbol{x})=f_1(\boldsymbol{x})^2\Delta_{d-1}(\boldsymbol{x})=f_1(\boldsymbol{x})^{2d-2}\prod_{1\leqslant i<j\leqslant d-1}(\alpha_i -\alpha_j)^2
\]
and in a similar way we find relations among the $d-1$ roots like those in \eqref{alphai}. Now we can split $\Sigma$ into finitely many subsets such that the $c_i$'s are fixed and that $f_0$ is either zero or non-zero for every point in a given subset. Arguing separately with each subset we may then assume we have $d$ (or $d-1$) values $c_i$ that do not depend on $P$.

We will handle these subsets in a different way depending on whether $f_0$ vanishes or not. We have already seen in the proof of Theorem \ref{main} how to proceed in the second case and we define a quasi-projective variety $W\subset\PP{n+4}$ just as before. On the other hand, if $f_0$ vanishes, the path is the same but we need to slightly modify the polynomials $A_1,\ldots,A_{d-1}$ and $B$ in an obvious way to deal with the fact that we have only $d-1$ roots. For $l=1,\ldots ,d-1$ we define
\begin{align*}
A_l'(Y_1,Y_2,Y_3,Y_4) & =\displaystyle{\sum_{1\leqslant i_1<\ldots <i_l\leqslant d-1}a_{i_1}(Y_1,Y_2,Y_3,Y_4)\cdots a_{i_l}(Y_1,Y_2,Y_3,Y_4)}\ \cdot\\
&\cdot\displaystyle{\prod_{1\leqslant j\leqslant d-1\; ,\; j\neq i_1,\ldots,i_l} b_j(Y_1,Y_2,Y_3,Y_4)}
\end{align*}
\[
B'(Y_1,Y_2,Y_3,Y_4)=\prod_{i=1}^{d-1} b_i(Y_1,Y_2,Y_3,Y_4).
\]

Then, we consider the projective variety $V'$ defined as the intersection of the zero loci of $f_0(x_0,\dots,x_n)$ and the $d-1$ polynomials
\[
B'(Y_1,Y_2,Y_3,Y_4)f_{l+1}(X_0,\ldots,X_n)-(-1)^{\, l} f_1(X_0,\ldots,X_n)A_l'(Y_1,Y_2,Y_3,Y_4)
\]
with $l$ ranging from $1$ to $d-1$. We define $U_0'$ exactly like $U_0$ in the proof of Theorem \ref{main} and we denote by $U_1'$ and $U_2'$, the set of points in $V$ with coordinates \mbox{$(z_0:\dots :z_n:y_1:\dots :y_4)$} such that, respectively, $B'(y_1,y_2,y_3,y_4)=0$ and $f_1(z_0,\dots ,z_n)=0$. As in the proof of the main theorem, we define a set $U':=U_0\cup U_1'\cup U_2'$ and its complement in $V$, the quasi-projective variety $W':=V'\setminus U'$. We notice that $W\cap W'=\varnothing$.

If we consider the projection $\pi\colon W\cup W'\to\PP{n}$ on the first $n+1$ coordinates, we observe that the subsets we have split $\Sigma$ in are contained either in $\pi(W)$ or in $\pi(W')$ and therefore $\Sigma\subset\pi(W\cup W')$. As in the proof of Theorem \ref{main}, we have $\pi\left( W\cap\{Y_2=Y_3\}\right )\subset T_0$ and, in a similar way, $\pi\left (W'\cap\{Y_2=Y_3\}\right )\subset T_1$. Remembering that $W\cap W'=\varnothing$ as well as $T_0\cap T_1=\varnothing$ we can conclude as follows:
\begin{align*}
\dim\overline{\Sigma} & \leqslant\dim\big(\pi(W)\cup\pi(W')\big)\\
&\leqslant\dim\Big(\big(\pi(W)\cup\pi(W')\big)\cap\{Y_2=Y_3\}\Big)+1\\
& = \dim\Big(\big(\pi(W)\cap\{Y_2=Y_3\}\big)\cup\big(\pi(W')\cap\{Y_2=Y_3\}\big)\Big)+1\\
& \leqslant\dim(T_0\cup T_1)+1.\qedhere
\end{align*}
\end{proof}\vspace{.3cm}

\noindent\emph{Low degrees -} We  now complete the proof of Theorem \ref{main} by taking into account the cases of $d=2$ and $d=3$. When $d=2$ we have two different roots $\alpha_1$ and $\alpha_2$ and we cannot apply results about $S$-unit equations: however we do not need them, since it is enough to use the trivial relations $\alpha_1=\alpha_1$ and $\alpha_2=\alpha_2$. Namely, we simply define the auxiliary polynomials $A_1(Y_1,Y_2)=Y_1+Y_2$ and $A_2(Y_1,Y_2)=Y_1 Y_2$; then we consider the variety $V\subset\PP{n+2}$ defined by the polynomials
\[
f_l(X_0,\ldots,X_n)-(-1)^{\, l} f_0(X_0,\ldots,X_n)A_l(Y_1,Y_2)\qquad l=1,2
\]
and the quasi-projective variety
\[
W:=V\setminus\big(\{f_0(X_0,\dots,X_n)=0\}\cup\{X_0=\dots =X_n=0\}.
\]
Subsequently, we consider the projection $\pi\colon W\to\PP n$ and everything will follow as in the proof of the main theorem: if $(x_0,\dots,x_n)\in\Sigma$ then $(x_0,\dots,x_n,\alpha_1,\alpha_2)\in W$ and the points in $\pi\big(W\cap\{Y_1=Y_2\}\big)$ satisfy the defining relations for $T_0$.

When $d=3$ the trivial relations considered above are no more sufficient to conclude, yet we lack the four different roots that enabled us to use results about the $S$-unit equation. If the hypersurface $\XX$ is defined in $\PP{n+1}$ by the polynomial
$$f(X_0,\dots,X_{n+1})=f_0X_{n+1}^3+f_1X_{n+1}^2+f_2X_{n+1}+f_3,$$where $f_i\in k[X_0,\dots,X_n]$ for i=0,1,2,3, let us suppose that $\deg f_0\geqslant 2$ and put $\delta:=\deg f_0-1$. We will introduce a subsidiary dimension and we consider the hypersurface $\mathcal Z\subset\PP{n+2}$ defined by
\[
g(X_0,\dots,X_n,Z,X_{n+1})=Z^\delta X_{n+1}^4+f_0X_{n+1}^3+f_1X_{n+1}^2+f_2X_{n+1}+f_3.
\]
We keep the notations previously introduced, adding a superscript $^\mathcal X$ or $^\mathcal Z$ when a definition is related to the hypersurface (or relative polynomial and projection map) considered. We approach the problem thinking that $H^\XX\simeq\PP n$, $H^\mathcal Z\simeq\PP{n+1}$ and $H^\XX\simeq H^\mathcal Z\cap\{Z=0\}$.

We arbitrarily choose a set of quasi-$S$-integral points $\Sigma$ for $H^\XX\setminus D^\XX$ and we must prove that the dimension of its Zariski closure is less than or equal to $\dim T_0^\XX+1$. For every point $(x_0:\dots:x_n)\in\Sigma\subset H^\XX$ we consider a point $(x_0:\dots:x_n:0)\in H^\mathcal Z$ and we denote the set of all these points by $\Sigma'$. Namely, we define$$\Sigma'=\{(x_0:\dots:x_n:0)\in H^\mathcal Z:(x_0:\dots:x_n)\in\Sigma\}.$$
It turns out that $\Sigma'$ is a set of quasi-$S$-integral points for $H^\mathcal Z\setminus\{\Delta^\mathcal Z=0\}$, for$$\Delta^\mathcal Z(x_0,\dots,x_n,0)=f_0(x_0,\dots,x_n)^2\cdot\Delta^\XX(x_0,\dots,x_n)$$and both factors on the right-hand term are non-zero because for every point $(x_0:\dots :x_n)\in\Sigma$ we have \mbox{$(x_0:\dots :x_n)\notin D^\XX$.} Now we parallel the proof of Corollary \ref{cordelta} to get that the dimension of the Zariski closure of $\Sigma'$ is less than or equal to $\dim T^\mathcal Z_1+1$. In fact, from the last displayed equation and denoting the roots of the polynomial $f$ by $\alpha_1,\alpha_2,\alpha_3$, we obtain
\begin{equation}\label{deltaz}
\begin{split}
\Delta^\mathcal Z(x_0,\dots,x_n,0)&=f_0(x_0,\dots,x_n)^6\prod_{1\leqslant i<j\leqslant 3} (\alpha_i-\alpha_j)^2\\
&=\prod_{1\leqslant i<j\leqslant 3} \big(f_0(x_0,\dots,x_n)(\alpha_i-\alpha_j)\big)^2
\end{split}
\end{equation}
which implies $(\alpha_i-\alpha_j)\in\OSU$ for every $i,j\in\{1,2,3\}$ such that $i\neq j$. This leads to a non-trivial fixed algebraic relation among the three roots and we outline the conclusion that follows similarly to the proof given for $d\geqslant 4$. In fact, $(\alpha_1-\alpha_2,\alpha_2-\alpha_3,\alpha_3-\alpha_1)$ is a non-degenerate solution of the $S$-unit equation $x_1+x_2+x_3=0$ and we can write $\alpha_3=(\alpha_1-\alpha_2)c+\alpha_1$ with $c$ in a finite set independent of the point chosen in $\Sigma$. Less auxiliary polynomials and only two variables are needed to define the variety $V$:
\[
a_1(Y_1,Y_2)=Y_1, \quad a_2(Y_1,Y_2)=Y_2,\quad a_3(Y_1,Y_2)=(Y_1-Y_2)c+Y_1
\]
\[
A_1(Y_1,Y_2)=a_1+a_2+a_3,\quad A_2(Y_1,Y_2)=a_1a_2+a_1a_3+a_2a_3
\]
\[
A_3(Y_1,Y_2)=a_1a_2a_3
\]
The variety $V$ is defined by the three polynomials
\[
f_l(X_0,\ldots,X_n)-(-1)^{\, l} f_0(X_0,\ldots,X_n)A_l(Y_1,Y_2)\quad l=1,2,3.
\]
and we define the variety $U$ as in the first part of this proof to obtain the quasi-projective variety $W$. As in the end of the proof of Corollary \ref{cordelta}, the dimension of the Zariski closure of the set $\Sigma'$ is less than or equal to $\dim T_1^\mathcal Z+1$. The sought conclusion follows observing that $\Sigma'\simeq\Sigma$ and $T^\mathcal Z_1\simeq T^\XX_0$.

We are left with the cases of $\deg f_0=0$ and $\deg f_0=1$ and we observe that the former has already a solution in \cite{zan}, since we can assume $f_0=1$ without loss of generality. If $\deg f_0=1$ we follow the proof given in this subsection for $\deg f_0\geqslant 2$ with the difference that the hypersurface $\mathcal Z\subset\PP{n+1}$ will be defined by the polynomial
\[
g(X_0,\dots,X_n,Z,X_{n+1})=ZX_{n+1}^4+f_0^2 X_{n+1}^3+f_0f_1X_{n+1}^2+f_0f_2X_{n+1}+f_0f_3.
\]
Everything goes as before with the exception of (\ref{deltaz}) that becomes
\begin{align*}\Delta^\mathcal Z(x_0,\dots,x_n,0)&=f_0(x_0,\dots,x_n)^8\Delta^\mathcal X(x_0,\dots,x_n)=\\&=f_0(x_0,\dots,x_n)^2\prod_{1\leqslant i<j\leqslant 3} \big(f_0(\alpha_i-\alpha_j)\big)^2\end{align*}
and the same conclusion follows.\qed

\end{document}